\documentclass[a4paper,oneside,11pt]{article}

\usepackage{amsxtra,amssymb,amsmath,amsthm,amsbsy,enumerate}
\usepackage[all,2cell,v2]{xy}\UseTwocells\UseHalfTwocells
\usepackage{tikz}

\usepackage[pdftitle={}, pdfauthor={}]{hyperref}

\theoremstyle{definition}
\newtheorem{theorem}{Theorem}[section]
\newtheorem{lemma}[theorem]{Lemma}
\newtheorem{remark}[theorem]{Remark}
\newtheorem{corollary}[theorem]{Corollary}
\newtheorem{example}[theorem]{Example}
\newtheorem{proposition}[theorem]{Proposition}

\def\id{{\rm id}}

\def\xto{\xrightarrow}
\def\to{\rightarrow}

\def\toto{\rightrightarrows}
\def\xfrom{\xleftarrow}

\def\action{\curvearrowright}
\def\xto{\xrightarrow}
\def\to{\rightarrow}

\def\toto{\rightrightarrows}
\def\xfrom{\xleftarrow}

\def\action{\curvearrowright}
\def\dashto{\dashrightarrow}
\def\then{\Rightarrow}

\def\im{{\rm im}}

\def\R{{\mathbb R}}

\begin{document}

\title{\bf The general linear 2-groupoid}


\author{Matias del Hoyo \and Davide Stefani}

\maketitle

\begin{abstract}
We deal with the symmetries of a (2-term) graded vector space or bundle. Our first theorem shows that they define a (strict) Lie 2-groupoid in a natural way. Our second theorem explores the construction of nerves for Lie 2-categories, showing that it yields simplicial manifolds if the 2-cells are invertible. Finally, our third and main theorem shows that smooth pseudo-functors into our general linear 2-groupoid classify 2-term representations up to homotopy of Lie groupoids. 
\end{abstract}

\tableofcontents


\section{Introduction}


A Lie group $G$ can be thought of as a smooth collection of symmetries of an abstract object. A linear representation $G\action V$ is therefore a way to realize these symmetries on a concrete vector space $V$, which we will assume to be finite dimensional and real. Such a representation can be defined either as a smooth map $\rho: G\times V\to V$ satisfying $\rho^h\rho^g=\rho^{hg}$ and $\rho^1=\id$, or as a Lie group morphism $G\to GL(V)$ into the general linear group. We can then study the group $G$ by looking at its representations $G\action V$, and this approach turns out to be very profitable. 


Following the previous philosophy, a Lie groupoid $G\toto M$ should be thought of as a smooth collection of symmetries of an abstract family parametrized by $M$. Lie groupoids have received much attention lately, as they provide a unifying framework for classic geometries, and also serve as models for spaces with singularities such as orbifolds and, more generally, differentiable stacks. The infinitesimal versions of Lie groupoids are Lie algebroids, geometric objects intertwining Lie algebra bundles and (singular) foliations. Differentiation and integration set up a fruitful interaction between the two theories.


A linear representation $(G\toto M)\action(V\to M)$ of a Lie groupoid over a vector bundle associates to each arrow \smash{$y\xfrom g x$} a linear isomorphism $\rho^g:V^x\to V^y$ between the corresponding fibers, in a way compatible with identities and compositions. It can be presented either as a partially defined map $G\times V\to V$ or as a Lie groupoid map $G\to GL(V)$ into the general linear groupoid \cite{dh}. 
The problem with Lie groupoid representations is that they are rather scarce, they impose strong conditions on $V$, and they do not provide us with enough information on $G\toto M$. This reflects in the lack of an adjoint representation, or in the limitations when establishing a Tannaka duality result for Lie groupoids (cf. \cite{t}).

\medskip


A solution for these problems was proposed by C. Arias Abad and M. Crainic, by introducing representations up to homotopy $G\action V$ of a Lie groupoid over a graded vector bundle \cite{ac}. They can be easily defined as differentials on certain bigraded algebras of sections, or alternatively, they can be regarded as a sequence of tensors: the first one is a differential $\partial$ on $V$, the second one consists of chain maps $\rho^g:V^x\to V^y$ between the fibers, the third one $\gamma^{h,g}$ provide chain homotopies relating $\rho^{hg}$ and $\rho^h\rho^g$, etc. Representation up to homotopy has proven to be a useful concept, for instance, when dealing with cohomology theory \cite{ac}, deformations \cite{cms} and Morita equivalences \cite{dho}.



When $V=V_1\oplus V_0$ is a 2-term graded vector bundle, a representation up to homotopy $G\action V$ leads to a VB-groupoid, a double structure mixing Lie groupoids and vector bundles, via a semi-direct product construction $G\ltimes V\to G$. It turns out that any VB-groupoid can be split as a semi-direct product, by choosing a horizontal lift of arrows, as proven first in \cite{gm}. This yields a 1-1 correspondence between VB-groupoids and 2-term representations up to homotopy, that can be extended to maps, and respect equivalence classes \cite{dho}. 
Prominent examples of VB-groupoids are the tangent and cotangent constructions. They encode the adjoint and coadjoint representations, respectively. 



A VB-groupoid is an instance of a fibration of groupoids, and according to classic Grothendieck correspondence, after choosing a horizontal lift of arrows, a groupoid fibration $E\to G$ is the same as a pseudo-functor $G\dashto\{\text{Groupoids}\}$ (cf. \cite{g}). It follows that 2-term representations up to homotopy should, in some sense, be the same as pseudo-functors. The main purpose of the present paper is to shed light on this. To take care of the smooth and the linear structure, we are led to fix a 2-term graded vector bundle $V$ and restrict our attention to pseudo-functors involving the several fibers of $V$. The resulting $G\dashto GL(V)$ is a suitable generalization of the classification map $G\to GL(V)$ for actual representations.


\medskip


Given $V=V_1\oplus V_0\to M$ a graded vector bundle, we construct a general linear 2-groupoid $GL(V)$, consisting of differentials on the fibers, quasi-isomorphisms between them, and chain homotopies.
There are several non-equivalent notions of Lie 2-groupoids in the literature, some of them too strict and some others too lax for our purposes. After discussing some variants, we introduce a notion of Lie 2-groupoid, and prove our Theorem \ref{thm:GL(V)}, asserting that $GL(V)$ is indeed a Lie 2-groupoid. It is remarkable that even for a 2-term graded vector space $V$ its general linear 2-groupoid $GL(V)$ is not a 2-group, it has more than one object, so groupoids arise naturally. 

In the set-theoretic context there is a nerve for 2-categories that relates lax functors with simplicial maps \cite{bfb,l}. We develop the smooth version of it, and our Theorem \ref{thm:NG} shows that, even though $NC$ is not always a simplicial manifold, it is so when the Lie 2-category $C$ has invertible 2-arrows, in particular for a Lie 2-groupoid.
This nerve construction relates our notion of Lie 2-groupoids with the simplicial approach to Lie 2-groupoids, based on the horn-filling condition, which has received much attention lately. This can be seen as a piece of evidence supporting our definitions for Lie 2-groupoids and smooth pseudo-functors. We also compare  our construction with that of \cite{mt}.


Building on the previous results, that we believe are of interest in their own, we finally establish our Theorem \ref{thm:main}, setting an equivalence of categories between 2-term representations up to homotopy $G\action V$ and pseudo-functors $G\dashto GL(V)$ commuting with basic projections. 
Combining this with the main theorem of \cite{gm}, and its extension in \cite{dho}, we get what we might call a smooth linear variant of Grothendieck correspondence (cf. \ref{rmk:grothendieck}):
$$\left\{\txt{VB-groupoids \\$\Gamma\to G$}\right\} 
\quad \rightleftharpoons \quad
 \left\{\txt{2-term RUTH \\ $G\action V_1\oplus V_0$}\right\}  \quad \rightleftharpoons \quad
 \left\{\txt{pseudo-functors \\ $G\dashto GL(V)$}\right\}$$

It seems natural to extend this result for higher degrees, relating positively graded representations up to homotopy and maps into a general linear $\infty$-groupoid.
Also, as potential applications of our theorem, we believe it is possible to relate our correspondence with the infinitesimal version announced in \cite{me}, and to frame the main theorem from \cite{dho} as a result about maps between differentiable 2-stacks. These problems will be explored elsewhere. 

\medskip


{\bf Organization. }
In sections 2 and 3 we quickly review 2-categories and their nerves, to fix notation and provide a reference for the tools needed later. Section 4 introduces our notion of Lie 2-groupoid and compares it with other important ones found in the literature. In section 5 we prove our first theorem, which constructs the fundamental example: the general linear 2-groupoid. Section 6 explores the combinatorics behind the nerve of 2-categories, and exploits it to establish our second theorem: the nerve of a Lie 2-category whose 2-cells are invertible is a simplicial manifold. In section 7 we prove our main theorem, realizing representations up to homotopies as maps, and we discuss further questions and applications. 

\medskip


{\bf Acknowledgements. } 
We thank H. Bursztyn, G. Ginot, A. Cabrera, E. Dubuc and M. Anel for many enriching conversations that shaped our view on the topic. 
We also thank C. Zhu for pointing out a mistake in a previous proof of Proposition \ref{prop:maximal}.
This collaboration took place during a series of visits of DS to IMPA, Rio de Janeiro, where MdH worked as a visiting professor. 
We thank IMPA for this opportunity, and Capes Cofecub and Projeto P.V.E. 88881.030367/2013-01 (CAPES/Brazil) for partial financial support.

\section{Basics on 2-categories}

\label{sec:basics}


We review here definitions and basic facts on set-theoretic 2-categories that are fundamental for the rest of the paper. We give a definition of 2-groupoid, compare it with others in the literature, and discuss the notion of lax functors. We refer to \cite{b,l,ma} for further details. The material here is preparatory, to set notation and conventions and to serve as a quick reference.

\medskip


A {\bf 2-category} $C$ is a category enriched over the category of small  categories. It has three levels of structure: objects, arrows between objects, and arrows between arrows or {\bf 2-cells}, whose collections we denote by $C_0, C_1, C_2$ respectively.
We use letters $x,y,\dots$ for objects, $f,g,\dots$ for arrows, and $\alpha,\beta,\dots$ for 2-cells.
$$\xymatrix{y \ar@{}[r]|{\Downarrow\alpha}& x \ar@/^0.7pc/[l]^{g} \ar@/_0.7pc/[l]_{f}}$$
The arrows and 2-cells between two fixed objects $x,y$ form a category $C(y,x)$, whose composition we denote by $\bullet$. For each triple $x,y,z$ there is a composition functor $C(z,y)\times C(y,x)\xto\circ C(z,x)$ and a unit $\id_x\in C(x,x)$ satisfying the axioms encoded in the following commutative diagrams:
{\footnotesize
$$\xymatrix@R=10pt@C=-30pt{
& \ar[ld]_(0.6){\circ\times\id} \ar[dr]^(0.6){\id\times\circ} C(w,z)\times C(z,y)\times C(y,x) & \\
C(w,y)\times C(y,x) \ar[dr]_\circ & & \ar[dl]^\circ  C(w,z)\times C(z,x)\\
& C(w,x) &}
 \qquad 
\xymatrix@R=10pt@C=10pt{
& \ar[rd]^{\id\times \id_x} \ar[dl]_{\id_y\times\id} \ar[dd]^\id C(y,x) & \\
C(y,y)\times C(y,x) \ar[dr]_\circ & & \ar[dl]^\circ  C(y,x)\times C(x,x)\\
& C(y,x) &} $$
}


\begin{example}
The paradigmatic example of a 2-category is that of small categories, functors and natural transformations. Another basic example is that of spaces, continuous maps and (homotopy classes of) homotopies. 
\end{example}


We are interested in 2-groupoids. For us, a {\bf 2-groupoid} $G$ is a 2-category such that (i) it is small, in the sense that $G_0$ is a set, (ii) every 2-cell is invertible, and (iii) every arrow \smash{$y\xfrom f x$} is invertible up to homotopy, namely there exist $x\xfrom g y$ and 2-cells $fg\cong\id_y$ and $gf\cong\id_x$. Some references demand the arrows be invertible on the nose. We call such 2-groupoids {\bf strict}. Let us remark that our fundamental example, that of the general linear 2-groupoid, is not strict. 

\begin{example}
A topological space $X$ yields a 2-groupoid $\pi_2(X)$ whose objects are the points of $X$, arrows are the continuous paths $I\to X$, and 2-cells are (homotopy classes of) path homotopies. Composition is given by juxtaposition, moving through each path at double speed. A non-constant path is only invertible up to homotopy, hence $\pi_2(X)$ is not strict.
\end{example}



A simple characterization of (small) 2-categories and strict 2-groupoids is by using {\em double structures}, namely diagrams of compatible structures as below, where compatible means that the horizontal structural maps are functorial with respect to the vertical structures. 
 $$\xymatrix{
G_2 \ar@<0.5ex>[r] \ar@<-0.5ex>[r] \ar@<0.5ex>[d] \ar@<-0.5ex>[d] & 
G_0 \ar@<0.5ex>[d] \ar@<-0.5ex>[d] \\
G_1 \ar@<0.5ex>[r] \ar@<-0.5ex>[r] & G_0}$$
However, our notion of 2-groupoid does not benefit much from this perspective. The following lemma, which is automatic for strict groupoids but works in general, will be useful later.

\begin{lemma}\label{lemma:factorization}
If $G$ is a 2-groupoid and $y\xfrom f x$ is an arrow in $G$, then the right multiplication functor $R_f:G(z,y)\to G(z,x)$ is an equivalence of categories for any $z$. The same holds for left multiplication.
\end{lemma}

\begin{proof}
A 2-cell $\alpha:f\then g$ defines a natural isomorphism $R_f\then R_g$, for the 2-cells are invertible. Then, given an arbitrary $f$, and picking $g$ a quasi-inverse, we have $\id_{G(x,x)}=R_{\id_x}\cong R_gR_f$ and analogously $\id_{G(y,y)}=R_{\id_y}\cong R_fR_g$.
\end{proof}



A {\bf functor} $\phi:C\to D$ between 2-categories consists of functions 
$\phi_i:C_i\to D_i$  preserving all the structure {\it on the nose}. This notion is sometimes too rigid for it involves many identities between functors. A useful variant is that of a (normal) {\bf lax functor} $\phi:C\dashto D$, which consists of three maps $\phi_i:C_i\to D_i$ preserving source, target, units and the composition $\bullet$, but only preserving $\circ$ up to a given natural transformation. More precisely, also given is a map
$$\phi_{1,1}:C_1\times_{C_0}C_1\to D_2 \qquad \phi_{1,1}(g,f):\phi_1(gf)\then\phi_1(g)\circ\phi_1(f)$$
ruling the failure of associativity of $\circ$ and satisfying these coherence axioms:
\begin{itemize}
 \item[i)] $\phi_{1,1}(\id_y,f)=\id_f=\phi_{1,1}(f,\id_x)$, where \smash{$y\xfrom f x$} (normality)
 \item[ii)] $(\phi_2(\beta)\circ \phi_2(\alpha))\bullet \phi_{1,1}(g,f) = \phi_{1,1}(g',f')\bullet \phi_2(\beta\circ\alpha)$, where  
 \smash{$\xymatrix{z \ar@{}[r]|{\Downarrow\beta} & y \ar@{}[r]|{\Downarrow\alpha} \ar@/^0.5pc/[l]^{g'} \ar@/_0.5pc/[l]_{g} & x \ar@/^0.5pc/[l]^{f'} \ar@/_0.5pc/[l]_{f}}$}
 \item[iii)] $(\phi_{1,1}(h,g)\circ \phi_1(f))\bullet \phi_{1,1}(hg,f)= (\phi_1(h)\circ \phi_{1,1}(g,f))\bullet \phi_{1,1}(h,gf)$, where 
 \smash{$w \xfrom{h} z \xfrom{g} y \xfrom{f} x$}
\end{itemize}
When the structure 2-cells $\phi_{1,1}(g,f)$ are invertibles the lax functor is called a {\bf pseudo-functor}. These notions are very interesting even when $C$ is a usual category, viewed as a 2-category with only identity 2-cells.
To ease the notation we will often write $\phi$ instead of $\phi_i$, etc.

\begin{example}\label{ex:grothendieck}
Given $\pi:G\to H$ an epimorphism of groups, a set-theoretic section $\sigma:H\to G$, $\sigma(1_G)=1_H$, leads to a pseudo-functor $\phi:H\dashto\{\rm Groups\}$, where $G$ is viewed as a 2-groupoid with one object and only identity 2-cells, and $\{\rm Groups\}$ is the 2-category of groups, morphisms, and inner automorphisms as 2-cells. Here $\phi(\ast)=K$ is the kernel of $\pi$, $\phi(h)$ is given by conjugation by $\sigma(h)$, and $\phi(h',h)$ is the conjugation by $\sigma(h')\sigma(h)\sigma(h'h)^{-1}$. The lax functor is an actual functor if and only if $\sigma$ is a morphism.
\end{example}

We also need to deal with morphisms between lax functors (cf. \cite{b}). Given $\phi,\psi:C\dashto D$ lax functors between 2-categories, a {\bf lax transformation}  $H:\phi\then\psi$ associates to each $x\in C_0$ an arrow $H_x:\phi(x)\to\psi(x)$ and to each arrow $f:x\to y$ a 2-cell $H_f:H_y\phi(f)\then \psi(f)H_x$ satisfying
\begin{itemize}
 \item[i)] $H_{\id_x}=\id_{H_x}$ (normality)
 \item[ii)] $(\psi(\alpha)\circ \id_{H_x})\bullet H_f=H_g\bullet(\id_{H_y}\circ\phi(\alpha))$, 
 where \smash{$\xymatrix{y \ar@{}[r]|{\Downarrow\alpha}& x \ar@/^0.5pc/[l]^{g} \ar@/_0.5pc/[l]_{f}}$}, and
 \item[iii)] for each pair of composable arrows \smash{$z\xfrom g y\xfrom f x$} there is a commutative prism with vertical faces $H_g,H_f,H_{gf}$ and horizontal faces given by the structural 2-cells of $\phi,\psi$.
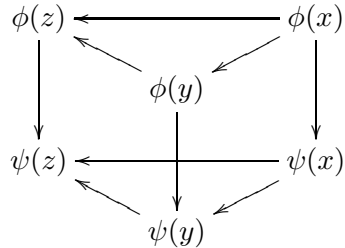
 $$\xymatrix@R=10pt{
\phi(z) \ar[dd]& & \phi(x)  \ar[ll] \ar[dd] \ar[ld]\\
 & \phi(y) \ar[dd] \ar[lu]& \\
\psi(z) & & \psi(x) \ar[ll] \ar[ld]\\
 & \psi(y) \ar[lu]&  
}$$
\end{itemize}
Such an $H$ is a {\bf lax equivalence} if the $H_x$ are invertible up to a 2-cell and the $H_f$ are invertible.

\begin{remark}\label{rmk:grothendieck}
Example \ref{ex:grothendieck} can be easily extended to suitable epimorphisms between categories, known as {\em fibred categories} 
(cf. \cite{b,g}). The outcome is {\em Grothendieck correspondence} between equivalence classes of fibred categories $E\to C$ and pseudo-functors $C\dashto\{\rm Categories\}$. This is the first and most important example of lax functors. The main goal of the present paper can be considered to be a smooth linear variant of this correspondence.
\end{remark}


\section{The nerve of a 2-category}


After reviewing the classic nerve construction, we discuss here the nerve for 2-categories and 2-groupoids. We explain its behavior with respect to lax functors, and  we use it to relate 2-groupoids with the weak approach to higher categories based on the horn filling condition. Some references for this are \cite{bc,bfb,h,l}.

\medskip


As usual, let $[n]=\{n,n-1,\cdots,1,0\}$ denote the ordinal of $n+1$ element, and let $\Delta$ be the category of finite ordinals and order preserving maps, spanned by the elementary maps
$$d^i:[n-1]\to [n] \quad d^i(k)=\begin{cases}
                           k & k<i\\ k+1 & k\geq i
                          \end{cases}\qquad
s^j:[n+1]\to [n] \quad s^j(k)=\begin{cases}
                           k & k\leq j\\ k-1 & k> j
                          \end{cases}$$
which satisfy the so-called simplicial identities. Then a {\bf simplicial set} is a contravariant functor $X:\Delta^\circ\to\{\rm Sets\}$. It can be described as a sequence of sets $X_n=X([n])$ and a collection of {\it face} $d_i=X(d^i)$ and {\em degeneracy} $s_j=X(s^j)$ operators satisfying the (dual) simplicial identities. Maps of simplicial sets are natural transformations, or equivalently, sequences of maps compatible with the faces and degeneracies. Simplicial objects on a category $\mathcal C$ are defined analogously.


\begin{example}
A simple but fundamental example is the {\bf $n$-simplex} $\Delta^n$. From the functorial viewpoint, it is the one represented by the ordinal $[n]$. Thinking of $\Delta^n$ as a graded set with further structure, it is freely generated by an element of type $[n]$, namely $\id_{[n]}$. By Yoneda's Lemma, a map $\Delta^n\to X$ is the same as an element in $X_n$. The {\bf border} $\partial\Delta^n\subset\Delta^n$ is spanned by all the faces of the generator, and the {\bf horn} $\Lambda^n_k\subset\Delta^n$ by all the faces but the $k$-th. 
\end{example}

Given $\mathcal C$ a category, and given $\phi:\Delta\to\mathcal C$ a covariant functor, which should be thought of as a model for simplices in $\mathcal C$, we can define a {\bf singular functor} $\phi^*:\mathcal C\to\{\text{Simplicial sets}\}$
that associates to each object $X\in\mathcal C$ a simplicial set by the formula $(\phi^*X)_n=\hom_{\mathcal C}(\phi([n]),X)$. 
In other words, $\phi^*X$ is the restriction of the contravariant functor represented by $X$ to $\Delta$ via $\phi$. 

\begin{example}
When $\mathcal C$ is the category of topological spaces and $\phi([n])$ is the topological $n$-simplex, then $\phi^*X=SX$ is the {\bf singular simplicial set} associated to $X$, used to define its homology. When $\mathcal C$ is the category of (small) categories and $\phi([n])=[n]$, where we see an ordinal as a category by setting an arrow $i\to j$ if $i\leq j$, then $\phi^*C=NC$ is the {\bf nerve} of the category,
whose $n$-simplices are chains of $n$ composable arrows, and whose faces and degeneracies are given by droping an extremal arrow, composing two consecutive ones, or inserting an identity. 
\end{example}


We are concerned with the nerve construction for 2-categories, namely the singular functor defined when $\mathcal C$ is the category of 2-categories and lax functors, and $\phi([n])=[n]$ is viewed as a 2-category with only identity 2-cells. Thus, if $C$ is a 2-category, then its {\bf nerve} $NC$ has as $n$-simplices the lax functors $u:[n]\dashto C$, and its simplicial operators are given by pre-composition. Note that $NC_0=C_0$ and $NC_1=C_1$ consist of the objects and arrows of $C$, respectively, and $NC_2$ consists of triangles that are commutative up to a given 2-cell:
$$ \xymatrix{ 
& y \ar[dl]_{g} & \\
z   & \ar@{}[u]|(0.4){\ \ \Uparrow \alpha} & x \ar[ul]_{f} 
\ar[ll]^{h}}$$


To describe the higher simplices, note that a lax functor $u:[n]\dashto C$ can be thought of as a labelling in an abstract $n$-simplex, where $u_i$ are objects at the vertices, $u_{j,i}$ are arrows at its edges, and $u_{k,j,i}$ are 2-cells corresponding to each triangle. For each tetrahedron on the simplex the following equation among 2-cells must hold:
$$\xymatrix@R=15pt@C=0pt{
& u_{l,i} \ar@{=>}[rd]^{u_{l,j,i}} \ar@{=>}[dl]_{u_{l,k,i}} & \\ 
u_{l,k}u_{k,i} \ar@{=>}[rd]_{u_{l,k} u_{k,j,i} \ \ \ \ } & & u_{l,j}u_{j,i} \ar@{=>}[dl]^{\ \ \ \ u_{l,k,j}{u_{j,i}}} \\
 & u_{l,k}u_{k,j}u_{j,i} &
}$$
The above data completely determines the nerve $NC$ in the sense that it is  {\em 3-coskeletal}, namely $NC_k=\{\partial\Delta^k\to NC\}$ for $k>3$. 
%


A fundamental feature of the classic nerve for 1-categories is that it defines a fully faithful functor, it embeds the category of (small) categories into that of simplicial sets. Extending this, there is the following proposition for the nerve of 2-categories, which also provides information about the 2-cells. Here, by a {\bf simplicial homotopy} we mean a simplicial map $X\times\Delta^1\to Y$.

\begin{proposition}[\cite{bfb}]\label{prop:lax=simplicial}
The nerve $C\mapsto NC$ defines a fully faithful functor from the category of (small) 2-categories and (normal) lax functors to the category of simplicial sets. Moreover, if $\phi,\psi:C\dashto D$ are lax functors and every 2-cell in $D$ is invertible, then there is a lax transformation $H:\phi\then\psi$ if and only if there is a simplicial homotopy $\tilde H:N\phi\cong N\psi$.
\end{proposition}

\begin{proof}[Sketch of proof]
Given a simplicial map $\tilde\phi:NC\to ND$, we can define a lax functor $\phi:C\dashto D$ such that $N\phi=\tilde\phi$ by setting $\phi_0=\tilde\phi_0$, $\phi_1=\tilde\phi_1$, and defining $\phi_2$ and $\phi_{1,1}$ as the restrictions of $\tilde\phi_2$ to the following types of triangles. The simplicial identities on $\tilde\phi$ imply the axioms of a lax functor on $\phi$, and that $N\phi=\tilde\phi$, proving the first assertion.
$$ \xymatrix{ 
& y \ar[dl]_{\id_y} & \\
y   & \ar@{}[u]|(0.4){\ \ \Uparrow \alpha} & x \ar[ul]_{f} 
\ar[ll]^{f}}
\qquad\qquad
\xymatrix{ 
& y \ar[dl]_{g} & \\
z   & \ar@{}[u]|(0.4){\ \ \Uparrow \id_{gf}} & x \ar[ul]_{f} 
\ar[ll]^{gf}}
$$ 
Regarding the second triangle, given $\phi,\psi:C\dashto D$ lax functors, while a lax transformation $H:\phi\cong\psi$ associates to an arrow \smash{$y\xfrom f x$} a 2-cell filling a commutative square, a simplicial homotopy $\tilde H:N\phi\cong N\psi$ should provide a triangulation of that square:
$$\xymatrix{
\phi(y) \ar[d]_{H_y} \ar@{}[drr]|{\ \ \Downarrow H_f} & & \phi(x) \ar[d]^{H_x} \ar[ll]_{\phi(f)} \\
\psi(y) & & \psi(x) \ar[ll]^{\psi(f)}}
\qquad \qquad
\xymatrix{
\phi(y) \ar[d]_{\tilde H_y} & \ar@{}[dl]|(0.3){\Uparrow \tilde H_{f,0} \ \ } & \phi(x) \ar[d]^{\tilde H_x} \ar[ll]_{\phi(f)} \ar[lld] 
\ar@{}[dl]|(0.7){\ \ \ \  \Downarrow \tilde H_{f,1}}\\
\psi(y) & & \psi(x) \ar[ll]^{\psi(f)}}$$
where $\tilde H_{f,0}$ and $\tilde H_{f,1}$ are short for $\tilde H(s_1(f),s_0(\id_{[1]}))$ and $\tilde H(s_0(f),s_1(\id_{[1]}))$.
The lax transformation $H$ induces a simplicial homotopy $\tilde H$ by setting $\tilde H_{f,0}=\id$ and $\tilde H_{f,1}=H_f$. Conversely, if every 2-cell on $D$ is invertible, we can define an $H$ out of $\tilde H$ by setting $H_f=\tilde H_{f,1}\bullet (\tilde H_{f,0})^{-1}$.
\end{proof}


Another fundamental feature of the classic nerve is the following characterization of its image: a simplicial set is the nerve of a category if and only if every inner horn ($0<k<n$) admits a filling, and this filling is unique for $n>1$. Similarly, it is the nerve of a groupoid if and only if the same holds for every horn, inner or not.
$$\xymatrix{\Lambda^n_k \ar[r]^\forall \ar[d] & X  \\ \Delta^n \ar@{-->}[ur]_{\exists(!)} &}$$
This motivates an approach to higher category theory that has received much attention lately. A simplicial set $X$ is then a {\bf weak $m$-category} if every inner horn in $X$ admits a filling, and the filling is unique for $n>m$, and $X$ is a {\bf weak $m$-groupoid} if the same holds for every horn, inner or not.
The missing face of the horn, provided by the filling, should be thought of as a {\em composition}, defined up to homotopy, of the remaining faces. The next proposition relates 2-groupoids with weak 2-groupoids via the nerve functor. Similar results are discussed in \cite{d}. 

\begin{proposition}\label{prop:bridge}
Given $C$ a 2-category, $NC$ is a weak 2-category if and only if every 2-cell of $C$ is invertible, and $NC$ is a weak 2-groupoid if and only if $C$ is a 2-groupoid.
\end{proposition}

\begin{proof}
Since $NC$ is 3-coskeletal, every $(n,k)$-horn has a unique filling for $n\geq 5$. For $n=2$ the horizontal composition of arrows provides inner horn-fillings, and the fillings of the outer horns correspond to the existence of quasi-inverses. So let us study the cases $n=3,4$.

For $n=3$, given a 2-cell $\alpha:f\then g:x\to y$, we can build a $(3,1)$-horn with faces thus,
$$\xymatrix@R=5pt@C=15pt{
& & y & & \\
& \ar@{}[rd]|{\id \then}& \ar@{}[rd]|{\id \then} & & \\
&  & y \ar[uu]^\id \ar[rrd]^\id & & \\
x \ar[rrrr]_f \ar[rruuu]^g \ar[rru]^g & &  \ar@{}[u]|(.4){\alpha\Uparrow\ } & & y\ar[lluuu]_\id
}$$
and the remaining face of a filling will give a right inverse $\beta:g\then f$ to $\alpha$, showing that inner-horn filling implies that every 2-cell is invertible. Conversely, a horn gives three 2-cells, which correspond to three sides of this square:
$$\xymatrix{
u_{3,0} \ar@{=>}[r] \ar@{=>}[d] & u_{3,1}u_{1,0} \ar@{=>}[d] \\
u_{3,2}u_{2,0} \ar@{=>}[r] & u_{3,2}u_{2,1}u_{1,0}
}$$
In an inner horn, either the 2-cell on the top or in the left is missing, but since every 2-cell is invertible, we can fill the square by taking the obvious composition. In an outer horn, either the 2-cell on the bottom or on the right is missing, and assuming $C$ is a 2-groupoid, we can get the missing face by factoring the triple composition by either $u_{3,2}$ or $u_{1,0}$ as it follows from \ref{lemma:factorization}.

For $n=4$, the 2-skeleton of a 4-simplex $u$ gives the edges of a cube as below:
$$
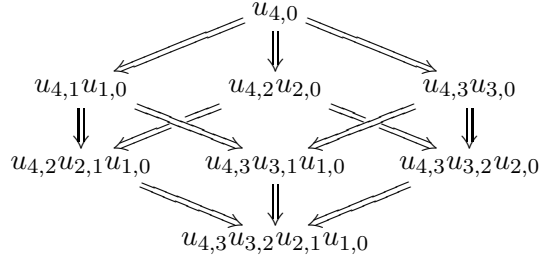
\xymatrix@R=15pt@C=5pt{
& u_{4,0} \ar@{=>}[rd] \ar@{=>}[dl] \ar@{=>}[d]& \\
u_{4,1}u_{1,0} \ar@{=>}[d] \ar@{=>}[dr] & 
u_{4,2}u_{2,0} \ar@{=>}[ld]|\hole \ar@{=>}[rd]|\hole & u_{4,3}u_{3,0} 
\ar@{=>}[d] \ar@{=>}[dl] \\
u_{4,2}u_{2,1}u_{1,0} \ar@{=>}[dr] & u_{4,3}u_{3,1}u_{1,0} \ar@{=>}[d] & 
u_{4,3}u_{3,2}u_{2,0} \ar@{=>}[dl] \\
& u_{4,3}u_{3,2}u_{2,1}u_{1,0}
}$$
Each face of the 4-simplex corresponds to the commutativity of the corresponding face of the cube. The bottom face commutes because of the compatibility between horizontal and vertical composition. Since every 2-cell is invertible, five commuting faces on the cube imply that the other is commutative as well, thus every horn admits a unique filling, concluding the proof.
\end{proof}


\begin{remark}\label{lemma:nerves}
Other ways to associate a simplicial set to a 2-category $C$ are by regarding it as a double category with a trivial side, applying twice the classic nerve, and reducing the resulting bisimplicial set by using the {\em diagonal} $d$ or the {\em total functor} $T$, also known as bar or codiagonal:
$$\text{2-categories}\xto{N^2}\text{bisimplicial sets}\overset{d,T}\toto\text{simplicial sets}$$
It is shown in \cite{bc} that $TN^2C$ and $dN^2C$ are equivalent to $NC$ from a homotopy viewpoint. We remark here that, when $C$ is a strict 2-groupoid there is actually an isomorphism $TN^2C\cong NC$, 
which is completely determined by the following formula for 2-cells. 
$$\begin{matrix}
\xymatrix@R=15pt@C=25pt{
z & \ar[l]_g y \ar@{=}[d] \ar@{}[dr]|{\ \Uparrow\alpha} &  \ar@{=}[d] x \ar[l]_h\\
& y  & x \ar[l]^f \ar@{=}[d] \\
 & & x}
  \end{matrix}
\quad\mapsto\quad
\begin{matrix}
\xymatrix{ 
& y \ar[dl]_{g} & \\
z   & \ar@{}[u]|(0.4){\ \ \Uparrow g\alpha^{-1}} & x \ar[ul]_{f} 
\ar[ll]^{gh}}
\end{matrix}$$
\end{remark}

\section{Defining Lie 2-groupoids}


We discuss here the smooth versions of 2-categories and 2-groupoids we are going to work with, provide some examples, and discuss other uses for those terms in the literature.

\medskip


A {\bf Lie 2-category} $C$ is, roughly speaking, a 2-category internal to the category of smooth manifolds. It consists of a (small) 2-category as defined before, on which (i) the sets of objects $C_0$, arrows $C_1$ and 2-cells $C_2$ are equipped with manifold structures; (ii) the source and target maps $s,t:C_i\to C_{i-1}$ of 2-cells and arrows are surjective submersions, and (iii) the units $u:C_{i-1}\to C_i$ and the multiplications $\circ$ and $\bullet$ are smooth.
Functors $\phi:C\to D$ between Lie 2-categories are easy to define, as 2-functors for which the three maps $\phi_i:C_i\to D_i$ are smooth. 


\begin{example}\label{example:deloopingR}
Let $(\R,\cdot)$ be the multiplicative monoid of real numbers, viewed as a Lie 2-category with a single object, space of arrows $\R$, and both horizontal and vertical composition equal to the multiplication. This is a Lie 2-category on which not every 2-cell is invertible.
\end{example}


Let $G$ be a Lie 2-category that, from the set-theoretic viewpoint, is also a 2-groupoid, as defined in the previous sections.
In order to define when $G$ is a Lie 2-groupoid we have to make sense of smooth inversions. For 2-cells this is clear, for there is an inversion map $i:G_2\to G_2$, and we can require it to be smooth. For arrows this is less clear, for inversion is only defined up to homotopy, there is no inversion map in general. Note that, since source and target $G_2\to G_1$ are surjective submersions, the sets of 2-horns $N_{2,i}G=\hom(\Lambda^2_i,NG)$ define manifolds:

$$N_{2,0}G=\left\{\begin{matrix}\xymatrix@R=5pt@C=5pt{
& y  & \\z   & & x \ar[ul]_{f}\ar[ll]^{h}}\end{matrix}\right\}
\qquad
N_{2,1}G=\left\{\begin{matrix}\xymatrix@R=5pt@C=5pt{ 
& y  \ar[ld]_g& \\ z   & & x \ar[ul]_{f}}\end{matrix}\right\}
\qquad
N_{2,2}G=\left\{\begin{matrix}\xymatrix@R=5pt@C=5pt{ 
& y \ar[ld]_g & \\ z   & & x  \ar[ll]^{h}}\end{matrix}\right\}
$$
We will discuss a smooth structure on the whole nerve $NG$ in the following sections. For now, we just endow $N_2G$ with a manifold structure using the fibered product
$$\xymatrix{
N_2G \ar[r] \ar[d] & N_{2,1}G \ar[d]^m \\
G_2 \ar[r]_t & G_1
}$$

We define $G$ to be a {\bf Lie 2-groupoid} if, besides being a Lie 2-category and a 2-groupoid, (i) the inversion of 2-cells $i:G_2\to G_2$ is smooth, and (ii) the following restriction maps are surjective submersions: 
$$d_{2,0}:N_2G\to N_{2,0}G \qquad d_{2,2}:N_2G\to N_{2,2}G.$$
We say that the Lie 2-groupoid is {\bf strict} if it is set-theoretic strict and the inversion of arrows $i:G_1\to G_1$ is smooth.
The smooth structure on $N_2G$ also allow us to make sense of lax functors in the smooth setting. We define a {\bf smooth lax functor} between Lie 2-categories $\phi:C\dashto D$ as a lax functor such that $\phi_0$, $\phi_1$ and the map $(\phi_2,\phi_{1,1}):N_2C\to N_2D$ are smooth.
A {\bf smooth lax transformation} $H:\phi\then\psi$ is one on which the maps $C_0\to D_1$, $C_1\to D_2$ are smooth.


\begin{example} 
Given $K$ an abelian Lie group, we can see it as the 2-cells of a Lie 2-category with one object and one arrow, and where both multiplications $\bullet$ and $\circ$ agree with that of $K$. The resulting 2-category $K\toto \ast\toto\ast$ is in fact a Lie 2-groupoid. A similar thing can be done with a bundle of abelian Lie groups $G\toto M$, such as a torus bundle. 
This {\em delooping} construction stays within the finite dimensional setting and plays a key role for instance in the theory of {\em gerbes}.
\end{example}


We would like to quickly review the {\bf Dold-Kan construction}. When $\mathcal C$ is an abelian category, eg. that of vector spaces, then a simplicial object $X:\Delta^\circ\to\mathcal C$ gives rise to a chain complex $(X'_n,\partial)$ by defining
$X'_n=\cap_{i>0} \ker(d_i:X_n\to X_{n-1})$ and $\partial=d_0$. It turns out that this construction yields an equivalence of categories between simplicial objects and positively graded chain complexes. The horn-filling condition translates into the abelian setting, in such a way that categories and groupoids both correspond to 2-term complexes, and linear natural transformations correspond to chain homotopies. 

\begin{example} 
By a {\bf linear 2-category} we mean a Lie 2-category $V$ on which the $V_i$ are (real finite dimensional) vector spaces and the structure maps are linear. They are examples of Lie 2-groupoids. Viewing them as double linear categories, and applying Dold-Kan correspondence both horizontally and vertically, we encode such a $V$ into a 3-term complex as follows: 
$$\xymatrix{
V_2' \ar[r] \ar[d] & 0 \ar[d]\\
V_1' \ar[r] & V_0
}$$
\end{example}


\begin{remark}
The term "Lie 2-groupoid" is used in the literature in senses other than the one we have introduced, which is suitable for our fundamental example. In \cite{mt} and other references, they use the word to refer to what we called strict Lie 2-groupoid, they demand the inverse of arrows to exists, our notion is more general. In \cite{z} and other references a Lie 2-groupoid is defined as a smooth version of weak 2-groupoids, they do not require the existence of a well-defined composition. We will see later that a smooth version of the nerve functor for Lie 2-categories allows us to regard our Lie 2-groupoids as examples of them.
\end{remark}


\section{The general linear 2-groupoid}


Here we show our first main theorem, asserting that the symmetries of a (2-term) graded vector space or bundle can be endowed with the structure of a Lie 2-groupoid, which we call the general linear 2-groupoid. This construction extends the general linear groupoid of a vector bundle without a grading (see eg. \cite{dh}).

\medskip


Throughout this section, let $V=V_1\oplus V_0\to M$ be a graded vector bundle over a smooth manifold. We will first describe the set-theoretic structure of its general linear 2-groupoid $GL(V)$ and then take care of the smoothness. From the set-theoretic viewpoint we have: 
\begin{itemize}
 \item[i)] An object $\partial^x\in GL(V)_0$ is a differential $\partial^x: V_1^x\to V_0^x$ on the fiber $V^x=V_0^x\oplus V_1^x$; 
 \item[ii)] An arrow $\alpha:\partial^x\to\partial^y\in GL(V)_1$ is a couple of linear maps $\alpha_1:V_1^x\to V_1^y$, $\alpha_0:V_0^x\to V_0^y$, defining a quasi-isomorphism between $V^x$ and $V^y$; 
$$\xymatrix{
V_1^x \ar[r]^{\alpha_1} \ar[d]_{\partial^x} & V_1^y \ar[d]^{\partial^y} \\
V_0^x \ar[r]_{\alpha_0} & V_0^y
}$$
 \item[iii)] A 2-cell $R:\alpha\to\alpha':\partial^x\to\partial^y$ on $GL(V)_2$ is a chain homotopy, given by a linear map $R:V_0^x\to V_1^y$ such that $R\partial^x=\alpha_1-\alpha'_1$ and $\partial^y R=\alpha_0-\alpha'_0$.
$$\xymatrix{
V_1^x \ar[r]^{\alpha_1} \ar[d]_{\partial^x} & V_1^y \ar[d]^{\partial^y} \\
V_0^x \ar[r]_{\alpha_0} \ar[ur]^R & V_0^y
}$$
\end{itemize}
The multiplication $\circ$ in $GL(V)$ is the composition of maps, and the multiplication $\bullet$ is the composition of chain homotopies, which is just the sum of the corresponding maps $R$. Every 2-cell is invertible, and every arrow is invertible up to a 2-cell. Thus we have a well-defined 2-groupoid $GL(V)$. Via Dold-Kan we can embed it into the 2-category of linear categories.

\begin{remark}
Even when $M=\ast$ our construction $GL(V)$ yields a 2-groupoid and not what one might call a 2-group, for there are many objects and not just one. Fixing an object $\partial$ on $GL(V)$, its isotropy 2-groupoid $GL(V)_\partial$ can be compared with the construction studied in \cite{sz}. 
\end{remark}


Next we show that $GL(V)$ inherits a smooth structure from certain vector bundles. To ease the notation, given $A,B\to M$ vector bundles, we write $[A,B]\to M$ for the inner-hom vector bundle.
Then:

\begin{itemize}
\item[i)] $GL(V)_0$ identifies with the total space of $[V_1,V_0]\to M$, 
\item[ii)] $GL(V)_1$ is a subspace of $E=[\pi_1^*V_1,\pi_1^*V_0]\oplus [\pi_2^*V_1,\pi_2^*V_0]\oplus
 [\pi_1^*V_1,\pi_2^*V_1]\oplus [\pi_1^*V_0,\pi_2^*V_0]$, a vector bundle over $M\times M$, where $\pi_i:M\times M\to M$ are the obvious projections, and 
 \item[iii)] $GL(V)_2$ is the set-theoretic fiber product $GL(V)_1\times_{M\times M} [\pi_1^*V_0, \pi_2^*V_1]$. 
\end{itemize}
 
The issue here is to show that $GL(V)_1\subset E$
is a submanifold. Then $GL(V)_2$ will identify with a fibered product along a submersion, in fact with a pullback vector bundle. This issue is rather subtle and will require a careful analysis. The first step in our argument is to provide a simple system of equations describing $GL(V)_1\subset E$.

\begin{lemma}\label{lemma:equations}
We can write $GL(V)_1=F\cap U_1\cap U_0$ where
\begin{align*}
F&=\{(\partial^x,\partial^y,\alpha_0, \alpha_1)\in E:\alpha_0\partial^x=\partial^y\alpha_1\}\\
U_1&=\{(\partial^x,\partial^y,\alpha_0,\alpha_1)\in E:\ker(\partial^x)\cap\ker(\alpha_1)=0\}\\
U_0&=\{(\partial^x,\partial^y,\alpha_0,\alpha_1)\in E:\im(\partial^y)+\im(\alpha_0)=V_0^y\}
\end{align*}
\end{lemma}

\begin{proof}
An element $(\partial^x,\partial^y,\alpha_0, \alpha_1)$ belongs to $F$ if and only if the corresponding square of vector space maps commutes, it belongs to $U_1$ if and only if the morphism between the fibers is injective in degree 1 homology, and belongs to $U_0$ if and only if it is surjective in degree 0 homology. Since both fibers $V^x,V^y$, as 2-term complexes, have the same Euler characteristic $\dim V_0-\dim V_1$, then so do their homologies, and therefore the two inequalities $\dim H_1(V^x)\leq\dim H_1(V^y)$ and $\dim H_0(V^x)\geq\dim H_0(V^y)$ imply that $\alpha$ is in fact a quasi-isomorphism.
\end{proof}


The subset $F$ can be seen as the preimage of the zero section of the following map between the total space of vector bundles over $M\times M$, where $E'=[\pi_1^*V_1,\pi_2^*V_0]$.
$$\phi:E\to E' 
\qquad 
\phi(\partial^x,\partial^y,\rho_1,
\alpha_0)=\alpha_0\partial^x-\partial^y\alpha_1$$
This map is quadratic and its rank is not constant in general, as the next example shows.

\begin{example}
Let  $M=\ast$ and $V_0=V_1=\R$. Then $GL(V)_0\cong\R$, $E\cong\R^4$ and $F$ identifies with $\{(x,y,z,w)\in\R^4:xy-zw=0\}$, which is not a submanifold of $\R^4$. This example shows that if we define the general linear 2-category $gl(V)$ as we have defined $GL(V)$, but without imposing the quasi-isomorphism axiom, then $gl(V)$ cannot be made a Lie 2-category in a reasonable way.  
\end{example}


Next we show that the map $\phi$ above has maximal rank over the opens $U_i$, and since the zero section $0_{M\times M}\subset E'$ is closed embedded, the same holds for $GL(V)_1=\phi^{-1}(0_{M\times M})\cap U_1\cap U_0$.

%
%
%

\begin{proposition}\label{prop:maximal}
If $p\in U_1\cup U_0$ and $q=\phi(p)$ then $d\phi_p:T_pE\to T_qE'$ is surjective.
\end{proposition}

\begin{proof}
Let $p=(\partial^x,\partial^y,\alpha_1,\alpha_0)\in U_1$ and let $q=\phi(p)=\alpha_0\partial^x-\partial^y\alpha_1$. The case $p\in U_0$ is analogous. Since $\phi$ commutes with the projections $E\to M\times M$ and $E'\to M\times M$, in order to prove that $d\phi_p:T_pE\to T_qE'$ is surjective, it is enough to show that any vertical vector $v\in T_qE'$ is in the image of $d\phi_p$. Take $v$ such a vertical vector and realize it as the 1-jet of a curve $\gamma$ within the fiber, $\gamma(t)\in E'_{(y,x)}$.
We need to construct a lifted curve $\tilde\gamma(t)\in E$ through $p$. 
Since $\ker\partial^x\cap\ker\alpha_1=0$, the linear map $(\partial^x,\alpha_1):V_1^x\to V_0^x\oplus V_1^y$ is a monomorphism, and it admits a linear retraction $(\tilde \partial^x,\tilde\alpha_1):V_0^x\oplus V_1^y\to V_1^x$. 
Finally, the curve
$$\tilde\gamma(t)=(\partial^x,\partial^y-(\gamma(t)-\gamma(0))\tilde\alpha_1,\alpha_1,\alpha_0+(\gamma(t)-\gamma(0))\tilde\partial^x)$$
satisfies $\phi(\tilde\gamma(t))=\gamma(t)$ and $\tilde\gamma(0)=p$, so it defines a lift, completing the proof.
\end{proof}

%



\begin{theorem}\label{thm:GL(V)}
Given $V=V_1\oplus V_0$ a graded vector bundle, its general linear 2-groupoid $GL(V)$ inherits a natural structure of a Lie 2-groupoid.
\end{theorem}

\begin{proof}
As we have already discussed, $GL(V)_0$ identifies $[V_1,V_0]$,  $GL(V)_1\subset E$ with the preimage of a closed embedded submanifold along a maximal rank map, and $GL(V)_2$ is a fiber product along a submersion.
It is straightforward to check that with these definitions the structure maps of $GL(V)$ are smooth, including the inversion of 2-cells. It only remains to show that the following restriction maps are surjective submersions:
$$d_{2,0}:N_2G\to N_{2,0}G \qquad d_{2,2}:N_2G\to N_{2,2}G$$
Let us show it for $d_{2,0}$, the other case is analogous. We argue again by lifting curves. 
We start with $\alpha(t):\partial^{x(t)}\to \partial^{y(t)}$ and $\gamma(t):\partial^{x(t)}\to \partial^{z(t)}\in GL(V)_1$, defining a curve on $N_{(2,0)}G$, and in order to lift it to $N_2G$, we want to define $\beta(t):\partial^{y(t)}\to\partial^{z(t)}$ and $R(t):\gamma(t)\then\beta(t)\alpha(t)$. 
Working locally we can again assume $x=x(t),y=y(t),z=z(t)$ are fixed. The monomorphism $(\alpha_1(t),\partial^x(t)):V^x_1\to V_1^y\oplus V_0^y$ admits a retraction $\tilde\alpha_1(t),\tilde\partial^x(t)$, and by basic arguments on linear algebra, we can take it smooth on $t$. 
Then the following short exact sequence splits smoothly,
$$0\to V_1^x \xto{(\alpha_1(t),\partial^x(t))} V_1^y\oplus V_0^y \xto{(\partial^y(t),\alpha_0(t))} V_0^y\to 0$$
and we gain a section $(\tilde\partial^y(t),\tilde\alpha_0(t))$. We can then define $\beta_i(t)=\gamma_i(t)\tilde\alpha_i(t)$ and $R(t)=\gamma_1(t)\tilde\partial^x$.
\end{proof}




%

\def\pair{{\rm Pair}}


\begin{remark}
Let us denote by $GL'(V)\subset GL(V)$ the open sub-2-groupoid with the same objects, arrows the invertible chain maps, and 2-cells the chain homotopies. This is a strict Lie 2-groupoid, somehow simpler than our version, and both agree around the units, thus both should behave in the same way with respect to {\em differentiation}, even though this process is not yet clear. See \cite{sz} for a related discussion. But regarding our purposes, this simpler construction $GL(V)'$ is not satisfactory, as there are representations up to homotopy of Lie groupoids which involve not invertible chain maps between the fibers (see Examples \ref{ex:tautological1} and \ref{ex:tautological2}).
\end{remark}


\section{The nerve of a Lie 2-category}


We deal here with the problem of endowing the nerve $NC$ of a Lie 2-category $C$ with a reasonable smooth structure. We show with a simple example that for general $C$ this may not be possible. Our second main theorem shows that if every 2-cell is invertible then $NC$ is indeed a simplicial manifold, and this happens for instance if $C$ is a Lie 2-groupoid. 

\bigskip


Given $C$ a Lie 2-category, we define its {\bf ambient} simplicial manifold $AC$ for the nerve $NC$, roughly speaking, by considering arbitrary collections $\{u_{k,j,i}\}$ of 2-cells and disregarding any compatibility. More precisely, we define $AC$ by
$$A_nC=\prod_{[2]\xto{a}[n]} C_2
\qquad
u\in A_nC,\ b:[m]\to[n]\ \then\ b^*(u)_a=u_{b\circ a}\in A_mC$$
This way $AC$ is a well-defined simplicial manifold, and every face map is a surjective submersion, for it is just the projection onto some of the coordinates.
There is a {\bf canonical inclusion} $\phi:NC\to AC$ defined by the formula $\phi(u)_{a} = (u\circ a)_{2,1,0}$, where $u\in N_nC$, $u:[n]\dashto C$, and $a:[2]\to[n]$. In other words, $\phi(u)$ keeps track of the 2-cells corresponding to each triangle, and by means of the identities, the arrows on the edges and the objects on the vertices. Since every simplex in $NC$ is determined by its 2-skeleton, the map $\phi$ is injective. We are concerned with the question of whether $\phi(N_nC)\subset A_nC$ is a submanifold, which is not the case in general.



\begin{example}
Let $(\R,\cdot)$ be the multiplicative monoid viewed as a Lie 2-category as described in Example \ref{example:deloopingR}.
Then $N_0C=\{\ast\}$, $N_1C=\{\id_\ast\}$, and $N_2C=\R$, but $N_3C\subset A_3C$ is not a submanifold. Disregarding the degenerate coordinates, we can identify $N_3C$ with tuples $(x,y,z,w)\in \R^4$ such that $xy=zw$, the equation corresponding  to the commutativity of the tetrahedron.  
\end{example}


For $C$ a 1-category, a simplex $u\in N_nC$ is the same as a chain of $n$ composable arrows, so we can write $N_nC$ as an iterated fiber product, and use this to define a smooth structure on it. Next we develop a similar combinatorial description for simplices $u\in N_nC$, where $C$ is a 2-category whose 2-cells are invertible. 

We see $\Delta^{n-1}$ inside $\Delta^{n}$ by using the face $d_n$, and define a decreasing filtration 
$$\Delta^n=F_0\Delta^n\supset F_1\Delta^n \supset\dots\supset F_{n-1}\Delta^n\supset \Delta^ {n-1}$$
by setting $F_k\Delta^n=\{a:[m]\to[n]/a(m)<n \text{ or } a(0)\geq k\}$, namely $F_k\Delta^n$ is the union of $\Delta^{n-1}$ with the last face of dimension $k$.
As an example, we depict the filtration for $n=3$:
$$
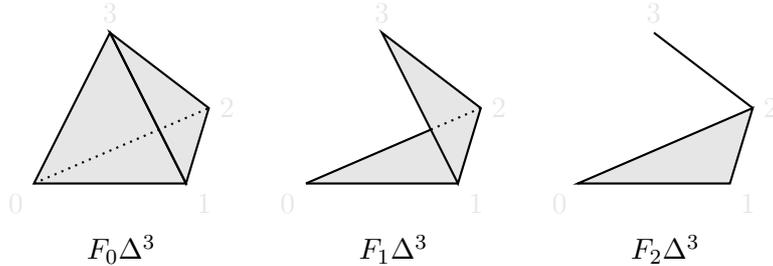
\begin{matrix}
\begin{tikzpicture}[thick,fill opacity=.1,text opacity=1]
\draw[fill](0,0) node[below left] {0}--(2,0) --(1,2)--cycle;
\draw[fill](1,2) node[above] {3}--(2,0) node[below right] {1} --(2.3,1) 
node[right] {2}--cycle ;
\draw[dotted](0,0)--(2.3,1);
\end{tikzpicture}
&
\begin{tikzpicture}[thick,fill opacity=.1,text opacity=1]
\draw[fill](5,2) node[above] {3}--(6,0) 
node[below right] {1} --(6.3,1) node[right] {2}--cycle ;
\fill (4,0) node[below left] {0}--(6,0)--(5.642,0.714)--cycle;
\draw(4,0)--(6,0);
\draw(4,0)--(5.642,0.714);
\draw[dotted](5.642,0.714)--(6.3,1);
\end{tikzpicture}
&
\begin{tikzpicture}[thick,fill opacity=.1,text opacity=1]
\draw[fill](8,0) node[below left] {0} --
(10,0) node[below right]{1}--(10.3,1) node[right]{2}--cycle;
\draw(10.3,1)--(9,2) node[above]{3};
\end{tikzpicture}
 \\
 F_0\Delta^3 & F_1\Delta^3 & F_2\Delta^3 
\end{matrix}
$$
Define $N^k_nC=\{F_k\Delta^n\to NC\}$. 
Note that $N^0_nC=N_nC$, that we have projections $N^k_nC\to N^{k+1}_nC$,  and that $N^{n-1}_nC=N_{n-1}\times_{C_0}C_1$ is the set-theoretic fiber product over $u\mapsto u_n$ and $s$. 


\begin{proposition}
If every 2-cell of $C$ is invertible then there are set-theoretic fiber products:
$$
\begin{matrix}\xymatrix{
\ar[r];[dr]^t \ar[d];[rd]_{\phi^{k}_n} \ar[r] \ar[d]
N^{k-1}_nC & C_2 \\ N^{k}_n C & C_1
}
\end{matrix}
\qquad
\phi^k_n(u)=u_{n,k}\circ u_{k,k-1}
$$
\end{proposition}

\begin{proof}


The inclusion $F_{k+1}\Delta^n\to F_{k}\Delta^n$ has all the vertices on its image, all the edges except for $(n,k)$, and all the triangles except for $(n,l,k)$, with $k<l<n$. Thus, given $u:F_{k}\Delta^n\to NC$, if we know its restriction $u'$ to $F_{k+1}\Delta^n$ and the 2-cell $\alpha$ corresponding to the triangle $(n,k+1,k)$, then we have all the vertices, we recover the edge $(n,k)$ as the source of $\alpha$, and we recover the  2-cells corresponding to the triangles $(n,l,k)$ inductively on $l-k$ by means of the equation:
$$u_{n,l,k} =
(u_{n,l}\circ u_{l,k+1,k})^{-1}\bullet(u_{n,l,k+1}\circ u_{k+1,k})\bullet 
u_{n,k+1,k}$$
This shows that the map $N^k_nC \to N^{k+1}_nC\times_{C_1}C_2$ is injective. 


To see that it is also surjective, we need to check that, given $u':F_{k+1}\Delta^n\to NC$ and given $\alpha:u'_{n,k}\then u'_{n,k+1}u'_{k+1,k}$, the above equations can be used to define a simplicial map $u:F_{k}\Delta^n\to NC$. This translates into showing that for every tetrahedron $(l,k,j,i)$ the above equation holds. The only tetrahedrons that deserve an explanation are those of the type $(n,l',l,k)$ with $k<l<l'<n$. Moreover, if $l=k+1$ then the equation holds by the construction of $u$. So let us assume that $k+1<l$. The 4-simplex corresponding to $(n,l',l,k+1,k)$ yields a cube as below:
%
$$\xymatrix{
& u_{n,k} \ar@{=>}[rd] \ar@{=>}[dl] \ar@{=>}[d]& \\
u_{n,k+1}u_{k+1,k} \ar@{=>}[d] \ar@{=>}[dr] & 
u_{n,l}u_{l,k} \ar@{=>}[ld]|\hole \ar@{=>}[rd]|\hole & u_{n,l'}u_{l',k} 
\ar@{=>}[d] \ar@{=>}[dl] \\
u_{n,l}u_{l,k+1}u_{k+1,k} \ar@{=>}[dr] & u_{n,l'}u_{l',k+1}u_{k+1,k} \ar@{=>}[d] & 
u_{n,l'}u_{l',l}u_{l,k} \ar@{=>}[dl] \\
& u_{n,l'}u_{l',l}u_{l,k+1}u_{k+1,k}
}$$
We want to see that the back right face commutes. But we know that: the back left face commutes by definition of $u_{n,l,k}$; the upper face commutes by definition of $u_{n,l',k}$; the left front face commutes for it factors through $u_{k+1,k}$; the right front face commutes for it factors through $u_{n,l'}$; and the bottom face commutes for $\circ$ and $\bullet$ are mutually distributive. Hence the result.
\end{proof}


We can now prove our second main theorem.

\begin{theorem}\label{thm:NG}
Given $C$ a Lie 2-category, if its 2-arrows are (smoothly) invertible, then its nerve $NC$ is naturally a simplicial manifold.
\end{theorem}

\begin{proof}
We endow each $N_nC$ with a smooth structure inductively. For $n=0,1$ we do this by means of the obvious identifications $N_0C=C_0$ and $N_1C=C_1$. For larger $n$ we use the filtration and fiber products of the previous proposition, noting that one of the maps is always a surjective submersion, and using the standard transversality criterion. Hence $N_nC$ is a closed embedded submanifold of the product
$$N_nC\subset N_{n-1}C\times \prod_{(i+1,i)}C_1\times\prod_{(n,i+1,i)} C_2$$
We will prove that, for these smooth structures, the canonical inclusion $\phi:N_nC\to A_nC$ into the ambient is a closed embedding. This implies that (i) the smooth structures that we have defined on $N_nC$ do not depend on the particular filtration we have used, and that (ii) the simplicial maps on $NC$ are smooth and $NC$ is a simplicial manifold.

For each triple $(k,j,i)$, we have to show that the composition
$\phi_{k,j,i}=\pi_{k,j,i}\phi:N_nC\to A_nC\to C_2$ is smooth. By projecting on the first coordinate of the above product, and using an inductive argument, we solve the case $n>k$. By projecting on the other coordinates we solve the cases $(n,i+1,i)$. It remains to study the other projections $\phi_{n,j,i}$. But such a projection can be written as an expression involving the other coordinates and the multiplications $\circ$ and $\bullet$, which are smooth. A similar argument applies also to the degenerate coordinates.
\end{proof}


It follows from our theorem that the nerve of a Lie 2-groupoid is a simplicial manifold, and that a smooth pseudofunctor $\phi:G\dashto G'$ is the same as a simplicial smooth map $\phi:NG\to NG'$. Next we present a less immediate corollary.

\begin{corollary}
With the above hypothesis, the face maps $d_i:N_nC\to N_{n-1}C$ are surjective submersions.
\end{corollary}
\begin{proof}
This is more a corollary of the proof rather than of the statement.
When $i=n$ it follows by factoring $d_n$ through the filtration, for each projection $N^k_nC\to N^{k+1}_nC$ is the base-change of a surjective submersion, as well as $N_n^{n-1}C\to N_{n-1}C$. 
When $i\neq n$ we can argue similarly, but now using a different filtration of $\Delta^n$, by complexes containing the face $d_i(\Delta^{n-1})$. 
\end{proof}


We finish this section by developing a smooth version of  \ref{prop:bridge}, setting a bridge between our theory and that of weak Lie $2$-categories and weak Lie $2$-groupoids, as defined in \cite{h,z}. A simplicial manifold $X$ is a {\bf weak Lie $m$-category} or a {\bf weak Lie $m$-groupoid} if the corresponding restrictions maps $X_{n}\to X_{n,k}$ are surjective submersions, for some reasonable smooth structure on the space of $(n,k)$-horns. 
The space of horns $X_{n,k}$ can be expressed as an equalizer
$$\prod_{i\neq k} X_{n-1}\toto \prod_{i,j\neq k} X_{n-1},$$
which may not exist in the category of manifolds. In general this is proved by an inductive argument. In our case, when $X=NC$ is the nerve of a Lie 2-category with invertible 2-arrows, it follows from our construction that
$X_{n}\to \prod_{i\neq k} X_{n-1}$
is a closed embedded submanifold for $n>3$ and for $n=3,k=2$. The case $n=3,k=1$ follows by using a symmetric filtration on the simplex. Therefore, since $X_n$ is also a set-theoretic equalizer, we conclude that $X_n\cong X_{n,k}$ is a diffeomorphism in these cases. The case $n=2$ is easy, and therefore we can conclude:

\begin{proposition}
Let $C$ be a Lie 2-category on which every 2-arrow is invertible.
Then $NC$ is a weak Lie 2-category. Moreover, $NC$ is a weak Lie 2-groupoid if and only if $C$ is a Lie 2-groupoid.
\end{proposition}


\begin{remark}
The main theorem on \cite{mt} shows that if $G$ is a strict Lie 2-groupoid then $TN^2G$ is a weak Lie 2-groupoid.
Thus, in light of the isomorphism described in \ref{lemma:nerves}, our theorem can be regarded as an extension of that to non-strict Lie 2-groupoids. This is crucial for us, for our fundamental example $GL(V)$ is not strict.
\end{remark}
%



\section{Representations as pseudo-functors}


In this section we review the notion of representation up to homotopy $G\action V$ of a Lie groupoid $G$, the particular case of 2-term vector bundles $V=V_1\oplus V_0$, and present our main theorem, stating a 1-1 correspondence between representations $G\action V$ and pseudo-functors $G\dashto GL(V)$. 

\medskip


Given $G\toto M$ a Lie groupoid and $E\to M$ a vector bundle, a {\bf representation} $G\action E$ can be defined as a map 
$\rho:G\times_M E\to E$, $\rho(y\xfrom g x, e)=\rho_g(e)$, such that (i)
$\rho_g:E_x\to E_y$ is linear, (ii) $\rho_{\id}=\id$, and (iii) $\rho_h\rho_g=\rho_{hg}$. A {\bf pseudo-representation} is a sort of non-associative action, it is defined analogously but just requiring (i) and (ii).


\begin{example}
If $G\toto \ast$ is a Lie group, viewed as a Lie groupoid with a single object, then its representations are the usual ones. If $M\toto M$ is a manifold, viewed as a Lie groupoid with only identities arrows, then its representations are the vector bundles over $M$. More generally, if $G\times M\toto M$ is the groupoid arising from a Lie group action $G\action M$, then a representation $(G\times M)\action E$ is the same as an equivariant vector bundle.
\end{example}

\def\hol{{\rm Hol}}

\begin{example}
Given $M$ a manifold, a representation $\pair(M)\action E$ of its pair groupoid is the same as a trivialization of $E$. Given a surjective submersion $q:M\to N$, a representation $M\times_N M\action E$ of the submersion groupoid (see \cite{dh}) is the same as an isomorphism $E\cong q^*E'$ with a pullback vector bundle. This can be further generalized to a foliation $F\subset TM$, which yields a holonomy groupoid $\hol(F)\toto M$, whose representations are the same as foliated bundles. 
\end{example}

\begin{example}\label{ex:tautological1}
Let $E\to P^2$ be the tautological line bundle over the real projective plane. Since it is not trivial there cannot be a representation of the pair groupoid $\pair(P^2)\action E$. Moreover, if $\rho:\pair(P^2)\action E$ is a pseudo-representation, then $\rho_g$ cannot be invertible for every $g$, otherwise it would still give a trivialization by transporting every vector to a fixed fiber. An example of such a pseudo-representation is mapping $\rho_{(\ell',\ell)}(v)$ to the orthogonal projection of $v\in\ell$ over $\ell'$. 
\end{example}


By means of the exponential law, a Lie groupoid representation can be described as a Lie groupoid morphism into the {\bf general linear groupoid} (see eg. \cite{dh})
$$\rho^\#:(G\toto M)\to (GL(E)\toto M) \qquad \rho^\#(g)=\rho_g$$
whose objects are the fibers of $E\to M$ and whose arrows are isomorphisms between fibers. In the case of a pseudo-representation we still have a smooth map $G\to GL(E)$ between the arrow spaces, compatible with source and target but that may fail to preserve the multiplication. 
This viewpoint allows one to treat representations as maps, and it is specially useful when dealing with differentiation and integration. 


Lie groupoid representations turn out to be very restrictive. A convenient generalization, is that of a {\bf representation up to homotopy} of a Lie groupoid $G$ over a graded vector bundle $V=\oplus V_i$. It is defined as a degree 1 differential $D$ on a space of sections $\Gamma(NG,V)$ of $V$ over the nerve of $G$ inducing a graded module structure. By decomposing $D=\oplus D_i$ into bi-homogeneous components,  we can reinterpret $D$ as a pseudo-representation over a complex $(V,\partial)$ with homotopies controlling its associativity. See \cite{ac,dho,mt} for further details. We recall here the 2-term case, the simplest new case, using an homological convention.




\begin{proposition}[\cite{dho,gm}]\label{prop:ruth}
If $V=V_1\oplus V_0$, then a representation up to homotopy $G\action V$ is the same as a tuple $(\partial,\rho_1,\rho_0,\gamma)$, where $\partial:V_1\to V_0$ is a linear map, $\rho_i:G\action V_i$ are pseudo-representations commuting with $\partial$, and \smash{$\gamma:({z\xfrom h y\xfrom g x})\mapsto(\gamma^{h,g}:\rho^{hg}\then\rho^{h}\rho^{g})$} is a curvature tensor satisfying
$$\rho_1^{g_3}\circ \gamma^{g_2,g_1}-\gamma^{g_3g_2,g_1}+\gamma^{g_3,g_2g_1}-\gamma^{g_3,g_2}\circ \rho_0^{g_1} = 0.$$
A morphism $\theta: V\to V'$ is the same as a triple $(\theta_1,\theta_0,\mu)$ where $\theta=(\theta_1,\theta_0):V\to V'$ is a vector bundle chain map and \smash{$\mu:({y\xfrom g x})\mapsto(\mu^{g}:V_0^x\to {V'_{1}}^y)$}
is a tensor satisfying 
$\rho'\theta-\theta\rho=\partial'\mu+\mu\partial$, and 
$$\theta^z_1\gamma^{h,g}+\mu^h\rho_0^g+{{\rho}'}_1^h\mu_g-\mu^{hg}-{\gamma}^{h,g}\theta^x_0=0.$$
\end{proposition}

\def\coker{{\rm coker\, }}

The {\bf point-wise homology} of a 2-term representation $G\action V$ consists of $H_1^x(V)=\ker \partial^x$ and $H_0^x(V)=\coker\partial^x$. If the rank of $\partial$ is constant then $H_1(V)$ and $H_0(V)$ are vector bundles and there is an induced representation over them.
A representation up to homotopy $V$ whose point-wise homology vanishes is called {\bf acyclic}. A morphism $\theta:V\to W$ of 2-term representations up to homotopy inducing isomorphims on the point-wise homology is called a {\bf quasi-isomorphism}.


\begin{example}\label{ex:tautological2}
For $\rho:\pair(P^2)\action E$ the pseudo-representation discussed before, we can define an acyclic representation up to homotopy $\pair(P^2)\action E\oplus E$ by setting $\partial=\id$, $\rho_1=\rho_0=\rho$ and $\gamma=\rho-\rho\rho$. The same can be done for any pseudo-representation.
\end{example}


\begin{example}
Given $G\toto M$ a Lie groupoid endowed with a connection $\sigma$, namely a section of $s:TG\to s^*TM$, the {\bf adjoint representation} $G\action (A\oplus TM)$ has $\partial$ equal to the anchor map and $\rho_0$ given by $t\sigma$. The equivalence class does not depend on $\sigma$. This generalizes the classical adjoint representation of Lie groups and plays a role in the deformation theory of groupoids. The {\bf coadjoint representation} $G\action T^*M\oplus A^*$ is defined by duality. 
\end{example}



%
%
%


We are now ready to present our main theorem. Given a Lie groupoid $G\toto M$ we have a canonical projection $\pi_G:G\to\pair(M)$ into the pair groupoid that just remembers the source and target of an arrow. Given a 2-term vector bundle $V\to M$, we have a canonical projection $\pi_V:GL(V)\to \pair(M)$ that only remembers the base-points on the vector bundle. 

\begin{theorem}\label{thm:main}
Given $G\toto M$ a Lie groupoid and $V=V_1\oplus V_0\to M$ a graded vector bundle, there is an equivalence between the category of representations up to homotopy $\rho:G\action V$ and quasi-isomorphisms and the category of pseudo-functors $\phi:G\dashto GL(V)$ satisfying $\pi_{V}\phi=\pi_G$ and smooth lax equivalences.
\end{theorem}

This result is truly a generalization of the situation for ordinary representations. That is, when $V$ is only in degree 0, then $GL(V)$ is the usual general linear groupoid, and the
pseudo-functors $G\dashto  GL(V )$ are just morphisms of Lie groupoids.

\begin{proof}
This is a rather direct consequence of the constructions and results collected during our work. In light of the set-theoretical simplicial interpretation (cf. \ref{prop:lax=simplicial}), our construction of the general linear 2-groupoid (cf. \ref{thm:GL(V)}), and our characterization for smooth nerve (cf. \ref{thm:NG}), a smooth pseudo-functor $\phi:G\dashto GL(V)$ is the same as a simplicial map $\phi:NG\to NGL(V)$. The degree 0 component $\phi_0$ is the same as a differential $\partial$ on $V$, the degree 1 component  $\phi_1$ gives a pseudo-representation $\rho$ on $V$ compatible with $\partial$, and the degree 2 component $\phi_2$ yields a curvature tensor \smash{$\gamma:({z\xfrom h y\xfrom g x})\mapsto(\gamma^{h,g}:\rho^{hg}\then\rho^{h}\rho^{g})$}, defining a 2-term representation up to homotopy, as characterized in proposition \ref{prop:ruth}. 
Similarly, a smooth lax equivalence $H:\phi\then\psi:G\dashto GL(V)$  consists of smooth maps $M\to GL(V)_1$, $G\to GL(V)_2$, corresponding to the components $\theta$ and $\mu$ of a quasi-isomorphism (cf. \ref{prop:ruth}). It is straightforward to check that these correspondences between objects and arrows are functorial.
%
%
\end{proof}

There are some remarks to be made regarding functoriality. Firstly, even though a quasi-isomorphism $\theta:V\to V$ of representations up to homotopy gives a simplicial homotopy $NG\times I\to NGL(V)$, not every such homotopy arises in this way, as can be seen in the proof of \ref{prop:lax=simplicial}. 
Secondly, if we want to consider morphisms $V\to V$ that are not quasi-isomorphisms, then the corresponding lax transformations would involve chain maps that are not within $GL(V)$.
Lastly, since the construction $V\mapsto GL(V)$ is not functorial, it makes little sense to frame the non-invertible morphism $V\to V'$ between different vector bundles within our theory.



We close this paper by outlining three different problems related to our results, the first related to the infinitesimal picture, the second to the theory of 2-stacks, and the third to higher versions of our results.


\begin{remark}
In \cite{me}, an infinitesimal analog to our main theorem was announced. 
It is commonly accepted that weak higher Lie groupoids and higher Lie algebroids are related by a theory of differentiation and integration, though the details of such a theory are yet to be understood. Within this context, we expect that the differentiation of our general linear 2-groupoid is the object $gl(V)$ introduced there, and that the differentiation and integration of maps will provide an alternative approach to the integration of 2-term representations up to homotopy, other than that of \cite{bcdh}.
\end{remark}


\begin{remark}
In \cite{dho}, the Morita equivalences of VB-groupoids are discussed. It is proved there that the derived category of VB-groupoids $VB[G]$ over a fixed base is a Morita invariant, and consequently, the same holds for 2-term representations up to homotopy. This result, from our framework, admits the following interpretation. Our general linear 2-groupoid $GL(V)$ represents a {\em differentiable 2-stack}, and the maps into it classify certain VB-groupoids, with prescribed side and core bundle. This should be thought of as an incarnation of the 2-stack $Perf_2$ appearing in algebraic geometry. Further details demand a better understanding of differentiable 2-stacks, and are postponed to be studied elsewhere.
\end{remark}


\begin{remark}
It is natural to expect our results to remain valid on higher degrees. The construction of the general linear groupoid seems suitable to be generalized for more general graded vector bundles. The understanding of pseudo-functors within this context seems to be less clear, though a complete immersion into the simplicial approach would solve this issue. Related to this, a realization of more general representations up to homotopy as higher VB-groupoids is being studied in the ongoing project \cite{dht}. Expectations here should be curbed, for even disregarding the smooth and linear structures, such a higher analog for Grothendieck correspondence is still unknown.
\end{remark}


\bigskip
 
\sf{\noindent Matias del Hoyo\\
Universidade Federal Fluminense (UFF),\\ 
Rua Professor Marcos Waldemar de Freitas Reis, s/n, 24210-201, Niter\'oi, Brazil.\\
mldelhoyo@id.uff.br}

\

\sf{\noindent Davide Stefani\\
Universit\'e  Pierre et Marie Curie (UPMC),  \\ 
5, place Jussieu, 75005, Paris, France.\\
davide.stefani@imj-prg.fr}

\end{document}